\documentclass[11pt,a4paper]{article}

\usepackage{graphicx,latexsym,euscript,makeidx,color,bm}
\usepackage{amsmath,amsfonts,amssymb,amsthm,thmtools,mathtools,mathrsfs,enumerate}
\usepackage[colorlinks,linkcolor=blue,citecolor=red]{hyperref}

\usepackage{geometry}
\allowdisplaybreaks[4]
   \def\cA{{\cal A}}  \def\hu{\hat{u}}
     \def\hv{\hat{v}}
\def\dbC{\mathbb{C}}   \def\cC{{\cal C}}  
     
\def\dbE{\mathbb{E}}     
\def\dbF{\mathbb{F}} \def\sF{\mathscr{F}}    
     
\def\dbH{\mathbb{H}}

   \def\cK{{\cal K}}  \def\bX{\bar{X}}
     \def\bu{\bar{u}}
 \def\sM{\mathscr{M}}

\def\dbP{\mathbb{P}}   
   \def\cQ{{\cal Q}}
\def\dbR{\mathbb{R}}   \def\cR{{\cal R}}
\def\dbS{\mathbb{S}}   \def\cS{{\cal S}}
 \def\sT{\mathscr{T}}  
 \def\sU{\mathscr{U}}  
 \def\sV{\mathscr{V}}  
   
  \def\hX{\hat{X}}

\def\BA{{\bf A}}
\def\BB{{\bf B}}
\def\BC{{\bf C}}
\def\BD{{\bf D}}

\def\BG{{\bf G}}

\def\BJ{{\bf J}}

\def\BP{{\bf P}}
\def\BQ{{\bf Q}}
\def\BR{{\bf R}}
\def\BS{{\bf S}}

\def\BV{{\bf V}}

\def\BX{{\bf X}}
\def\BY{{\bf Y}}
\def\BZ{{\bf Z}}

\def\Bb{{\bf b}}

\def\Bg{{\bf g}}

\def\Bq{{\bf q}}
\def\Br{{\bf r}}

\def\Bu{{\bf u}}

\def\BTh{\boldsymbol\Theta}

\def\BBsi{\boldsymbol\sigma}

\def\BBf{\boldsymbol\varphi}

\def\ns{\noalign{\ss}}
\def\ds{\displaystyle}
\def\ss{\smallskip}   \def\lt{\left}          \def\hb{\hbox}
\def\ms{\medskip}     \def\rt{\right}         \def\ae{\text{a.e.}}
\def\h{\widehat}          \def\lan{\langle}       \def\les{\leqslant}
\def\ges{\geqslant}
\def\as{\text{a.s.}}
\def\q{\quad}         \def\ran{\rangle}       \def\tr{\hb{tr$\,$}}
\def\qq{\qquad}             
\def\no{\noindent}          
\def\hp{\hphantom}    \def\blan{\big\lan}     \def\sc{\scriptscriptstyle}
\def\nn{\nonumber}    \def\bran{\big\ran}     \def\scT{\sc T}
\def\rf{\eqref}       \def\Blan{\Big\lan\!\!} \def\ds{\displaystyle}
\def\cd{\cdot}        \def\Bran{\!\!\Big\ran} 
\def\deq{\triangleq}  \def\({\Big(}           
       \def\){\Big)}           \def\tb{\textcolor{blue}}
   \def\[{\Big[}           
  \def\]{\Big]}
\def\({\Big (}
\def\){\Big )}
\def\[{\Big[}
\def\]{\Big]}
\def\3n{\negthinspace \negthinspace \negthinspace }
\def\2n{\negthinspace \negthinspace }
\def\1n{\negthinspace }
\def\bel{\begin{equation}\label}
\def\ee{\end{equation}}
\def\bea{\begin{eqnarray}}
\def\eea{\end{eqnarray}}
\def\bt{\begin{theorem}\label}
\def\et{\end{theorem}}
\def\bc{\begin{corollary}\label}
\def\ec{\end{corollary}}
\def\bex{\begin{example}\label}
\def\ex{\end{example}}
\def\bl{\begin{lemma}\label}
\def\el{\end{lemma}}
\def\bp{\begin{proposition}\label}
\def\ep{\end{proposition}}
\def\br{\begin{remark}\label}
\def\er{\end{remark}}
\def\ba{\begin{array}}
\def\ea{\end{array}}
\def\bde{\begin{definition}\label}
\def\ede{\end{definition}}

\def\a{\alpha}       \def\l{\lambda}    
\def\b{\beta}               \def\F{\varPhi}
\def\d{\delta}           
      
\def\f{\varphi}             \def\Om{\varOmega}
       \def\si{\sigma}    \def\Si{\varSigma}
\def\i{\infty}             \def\Th{\varTheta}

\newtheoremstyle{thry}
{}      
{}      
{\sl}   
{}      
{\bf}   
{.}     
{.5em}  
{}      
\theoremstyle{thry}

\newtheorem{theorem}{Theorem}[section]
\newtheorem{proposition}[theorem]{Proposition}
\newtheorem{corollary}[theorem]{Corollary}
\newtheorem{lemma}[theorem]{Lemma}

\theoremstyle{definition}
\newtheorem{definition}[theorem]{Definition}
\newtheorem{example}[theorem]{Example}

\theoremstyle{remark}
\newtheorem{remark}[theorem]{Remark}

\def\punct{}
\newtheoremstyle{dotless}{}{}{\rm}{}{\bf}{\punct}{.5em}{}
\theoremstyle{dotless}
\newtheorem{taggedthmx}{}
\newenvironment{taggedthm}[1]
 {\renewcommand\thetaggedthmx{#1}\taggedthmx}
 {\endtaggedthmx}
\newtheorem{taggedassumptionx}{}
\newenvironment{taggedassumption}[1]
 {\renewcommand\thetaggedassumptionx{#1}\taggedassumptionx}
 {\endtaggedassumptionx}

\makeatletter
   
   \@addtoreset{equation}{section}
   \newcommand{\setword}[2]{%
   \phantomsection
   #1\def\@currentlabel{\unexpanded{#1}}\label{#2}%
   }
\makeatother

\begin{document}

\title{\bf Turnpike Properties for Stochastic Linear-Quadratic Optimal Control Problems}

\author{Jingrui Sun\thanks{Department of Mathematics, Southern University of Science and Technology, Shenzhen,
                           518055, China (Email: {\tt sunjr@sustech.edu.cn}).
                           This author is supported by NSFC grant 11901280 and Guangdong Basic and Applied Basic
                           Research Foundation 2021A1515010031.}
~~~
Hanxiao Wang\thanks{Department of Mathematics, National University of Singapore, Singapore 119076, Singapore
                    (Email: {\tt hxwang14@fudan.edu.cn}).
                    This author is supported by Singapore MOE AcRF Grant R-146-000-271-112.}
~~~
Jiongmin Yong\thanks{Department of Mathematics, University of Central Florida, Orlando, FL 32816, USA
                    (Email: {\tt jiongmin.yong@ucf.edu}).
                    This author is supported by NSF grant DMS-1812921.}
}

\maketitle

\vskip-0.5cm

\centerline{({\sl In the Memory of Professor Chaohao Gu})}

\ms

\ms

\no{\bf Abstract.}
This paper analyzes the limiting behavior of stochastic linear-quadratic optimal control problems in finite time horizon $[0,T]$ as $T\to\infty$. The so-called turnpike properties are established for such problems, under stabilizability condition which is weaker than the controllability, normally imposed in the similar problem for ordinary differential systems. In dealing with the turnpike problem, a crucial issue is to determine the corresponding static optimization problem.
Intuitively mimicking deterministic situations, it seems to be natural to include both the drift and the diffusion as constraints in the static optimization problem. However, this would lead us to a wrong direction. It is found that the correct static problem should contain the diffusion as a part of the objective function, which reveals a deep feature of the stochastic turnpike problem.

\ms
\no{\bf Keywords.}
Turnpike property, stochastic optimal control, static optimization, linear-quadratic,
stabilizability, Riccati equation.

\ms
\no{\bf AMS 2020 Mathematics Subject Classification.}  49N10, 93D23, 93E15, 93E20.

\section{Introduction}\label{Sec:Intro}

Let $(\Om,\sF,\dbP)$ be a complete probability space on which a standard one-dimensional Brownian
motion $W=\{W(t)\bigm|t\ges0\}$ is defined.
Denote by $\dbF=\{\sF_t\}_{t\ges0}$ the usual augmentation of the natural filtration generated by $W$.
For a random variable $\xi$, we write $\xi\in\sF_t$ if $\xi$ is $\sF_t$-measurable;
and for a stochastic process $X$, we write $X\in\dbF$ if it is progressively measurable with respect
to the filtration $\dbF$.

\ms

Consider the following controlled linear stochastic differential equation (SDE, for short)
\bel{state}\left\{\2n\ba{ll}
\ds dX(t)=[AX(t)+Bu(t)+b]dt+[CX(t)+Du(t)+\si]dW(t),\q t\ges0,\\
\ns\ds X(0)=x,\ea\right.\ee
and the following general quadratic cost functional
\bel{TP:cost}\ba{ll}
\ns\ds J_{\scT}(x;u(\cd))={1\over2}\dbE\int_0^T\[\lan QX(t),X(t)\ran+2\lan SX(t),u(t)\ran+\lan Ru(t),u(t)\ran\\
\ns\ds\qq\qq\qq\qq\qq+2\lan q,X(t)\ran+2\lan r,u(t)\ran\]dt,\ea\ee
where $A,C\in\dbR^{n\times n}$, $B,D\in\dbR^{n\times m}$, $Q\in\dbS^n$, $S\in\dbR^{m\times n}$, $R\in\dbS^m$, $b,\si,q\in\dbR^n$, and $r\in\dbR^m$ are constant matrices or vectors with $\dbS^k$ being the set of all $(k\times k)$ symmetric matrices. The classical {\it stochastic linear-quadratic {\rm(LQ, for short)} optimal control problem} over the finite time-horizon $[0,T]$ is to find a control $\bu_{\scT}(\cd)$ from the space
\begin{align}\label{def:cUT}
\sU[0,T] = \lt\{ u:[0,T]\times\Om\to\dbR^m \bigm| u\in\dbF~\text{and}~\dbE\int_0^T|u(t)|^2dt<\i \rt\}
\end{align}
such that the cost functional \rf{TP:cost} is minimized over $\sU[0,T]$, for a given initial state $x\in\dbR^n$. More precisely, it can be stated as follows.

\begin{taggedthm}{Problem (SLQ)$_{\scT}$.}
For any given initial state $x\in\dbR^n$, find a control $\bu_{\scT}(\cd)\in\sU[0,T]$ such that
\begin{align}\label{TP:opt-u}
J_{\scT}(x;\bu_{\scT}(\cd)) = \inf_{u(\cd)\in\sU[0,T]} J_{\scT}(x;u(\cd)) \equiv V_{\scT}(x).
\end{align}
\end{taggedthm}

The process $\bu_{\scT}(\cd)$ in \rf{TP:opt-u} (if exists) is called an {\it open-loop optimal control}
of Problem (SLQ)$_{\scT}$ for the initial state $x$,
the corresponding state process $\bX_{\scT}(\cd)$ is called an {\it open-loop optimal state process},
$(\bX_{\scT}(\cd),\bu_{\scT}(\cd))$ is called an {\it open-loop optimal pair}, and $V_{\scT}(\cd)$ is called the {\it value function} of Problem (SLQ)$_{\scT}$.

\ms

In this paper, we are concerned with the limiting behavior of the optimal pair $(\bX_{\scT}(\cd)$, $\bu_{\scT}(\cd))$ of Problem (SLQ)$_{\scT}$ as the time-horizon $T$ tends to infinity. More precisely, we want to seek conditions under which there exist positive constants $K,\mu>0$, independent of $T$, such that for some $(x^*,u^*)\in\dbR^n\times\dbR^m$, it holds
\bel{Unibound:EXEu}\big|\dbE[\bX_{\scT}(t)-x^*]\big|+
\big|\dbE[\bu_{\scT}(t)-u^*]\big| \les K\big[e^{-\mu t} + e^{-\mu(T-t)}\big], \q\forall t\in[0,T].\ee
This is referred to as the {\it exponential turnpike property} of Problem (SLQ)$_{\scT}$. Such a property implies that for any small $\d\in(0,1/2)$, the following is true:
\bel{d,t-d}\big|\dbE[\bX_{\scT}(t)-x^*]\big|+\big|\dbE[\bu_{\scT}(t)
-u^*]\big|\les2Ke^{-\mu\d T},\qq\forall t\in[\d T,(1-\d)T].\ee
Namely, in a big portion $[\d T,(1-\d)T]$ of $[0,T]$, the optimal pair $(\bar X_{\scT}(\cd),\bar u_{\scT}(\cd))$ is exponentially close to the point $(x^*,u^*)$. This will give us the essential picture of the optimal pair without having to solve it analytically, which is very useful in applications.

\ms

The study of turnpike phenomena for deterministic problems can be traced back to the work of von Neumann \cite{Neumann1945} on problems in economics. In 1958, Dorfman, Samuelson, and Solow (\cite{Dorfman-Samuelson-Solow1958}) coined the name ``turnpike'' which was used in the highway system of the United States. Since then the turnpike phenomena have attracted considerable attentions, not only in mathematical economy \cite{McKenzie1976}, but also in many other fields such as mathematical biology \cite{Ibane2017} and chemical processes \cite{Rawlings-Amrit2009}. It is well-known by now that the turnpike property is a general phenomenon which holds for a large class of variational and optimal control problems. Numerous relevant results have been established for finite and infinite dimensional problems in the context of deterministic discrete-time and continuous-time systems (see, e.g., \cite{Carlson-Haurie-Leizarowitz1991,Zaslavski2006,Zaslavski2011,Trelat-Zuazua2015,Zuazua2017,Trelat-Zhang2018,
Lou-Wang2019,Zaslavski2019,Breiten-Pfeiffer2020,Grune-Guglielmi2021} and the references therein). In particular, we mention the papers \cite{Damm-Grune-Stieler-Worthmann2014,Grune-Guglielmi2018} for discrete-time LQ problems and the papers \cite{Porretta-Zuazua2013,Porretta-Zuazua2016} for continuous-time LQ problems of ordinary differential equations.

\ms

The study of turnpike phenomena for stochastic optimal control problems is quite lacking in literature. In this paper, we shall carry out a thorough investigation on the turnpike property for the stochastic LQ optimal control problem introduced earlier.
Note that when $C=0$ and $D=0$, Problem (SLQ)$_{_T}$ reduces to a deterministic LQ problem, for which the exponential turnpike property has been established in \cite{Porretta-Zuazua2013}
and \cite{Trelat-Zuazua2015} under controllability and observability assumptions. For the deterministic LQ problem (i.e., the case of $C=0,D=0,\si=0$), the associated static optimization problem, which is used to determine the point $(x^*,u^*)$, reads
\bel{DLQ:static}\left\{\2n\ba{ll}
\ds\text{Minimize}~F_0(x,u)\1n\equiv\1n\lan Qx,x\ran\1n +\1n2\lan Sx,u\ran\1n+\1n\lan Ru,u\ran+2\lan q,x\ran+2\lan r,u\ran,\\
\ds\text{subject to}\q Ax+Bu+b=0.\ea\right.\ee
To establish the turnpike property for the stochastic LQ problem, suggested by the deterministic situation, one might naively introduce the following static optimization problem:
\bel{Fake}\left\{\2n\ba{ll}
\ds \text{Minimize}\q F_0(x,u),\\
\ds \text{subject to}\q Ax+Bu+b=0,\q Cx+Du+\si=0.\ea\right.\ee
Assume the above admits an optima solution $(x^*,u^*)$. Then one tries to show that the optimal pair $(\bar X_{\scT}(\cd),\bar u_{\scT}(\cd))$ of Problem (SLQ)$_{\scT}$ satisfies \rf{Unibound:EXEu}. However, a little careful observation of the above, one immediately realize it is not natural because the condition that ensuring such an optimization problem to be feasible is already very restrictive: The two equality constraints might be contracting each other. It turns out that \rf{Fake} is not the correct one, which will be shown later in this paper.
As a main contribution of this paper, we found that the correct formulation of the static optimization problem is as follows:
\bel{static}\left\{\2n\ba{ll}
\ds\hb{Minimize}\q F(x,u)\equiv F_0(x,u)+\lan P(Cx+Du+\si),Cx+Du+\si\ran,\\
\ns\ds\hb{subject to }\q Ax+Bu+b=0,\ea\right.\ee
where $P$ is a positive definite solution to the following algebraic Riccati equation (ARE, for short):
\bel{ARE}\ba{ll}
\ds PA+A^\top P+C^\top PC+Q\\
\ds\q-(PB+C^\top PD+S^\top)(R+D^\top PD)^{-1}(B^\top P+D^\top PC+S)=0.\ea\ee
By assuming the controlled homogenous state equation (denoted by $[A,C;B,D]$) to be {\it stabilizable} (see the next section for a precise definition) and the following {\it strong standard condition}:
\bel{R>0}R>0,\q Q-S^\top R^{-1}S>0,\ee
one will have a unique suitable positive definite solution $P$ to the above ARE \rf{ARE}, and problem \rf{static} is not only feasible, but also admits a unique solution $(x^*,u^*)$. We will show that there exist positive constants $K,\mu>0$, independent of $T$, such that \rf{Unibound:EXEu} holds, and the adjoint process $\bar Y_{\scT}(\cd)$ will also have the same turnpike property. Note that by a (classical) standard condition, we mean that $Q-S^\top R^{-1}S$ is merely positive semi-definite, which could even be 0. In such cases, $P$ might not be positive definite, and $(x,u)\mapsto F(x,u)$ might not be coercive. Therefore, it is unclear if the optimal solution $(x^*,u^*)$ exists, or it might not be unique. This might bring some additional issues into the study and we will try to address that in our future publications. Also, we note that for the state equation, stabilizability is strictly weaker than the (null) controllability which was assumed in \cite{Porretta-Zuazua2013,
Trelat-Zuazua2015,Lou-Wang2019} for deterministic problems. For the study of controllability of linear SDEs, see \cite{Wang-Yang-Yong-Yu2017}.

\ms

The rest of the paper is organized as follows. In \autoref{Sec:Pre}, we give the preliminaries and collect some relevant results on stochastic LQ optimal control problems.
In \autoref{Sec:Stability} we recall the notion of stabilizability and formulate the correct static optimization problem. The convergence of the solution to a related differential Riccati equation as the time-horizon tends to infinity will be presented in \autoref{Sec:Convergence}. In \autoref{Sec:TP}, we study the static optimization problem associated to Problem (SLQ)$_{\scT}$ and establish the turnpike property of Problem (SLQ)$_{\scT}$ as well as of the adjoint process. Some concluding remarks are collected in \autoref{Sec:Concludng}.

\section{Preliminaries}\label{Sec:Pre}

We begin with some notation that will be frequently used in the sequel.
Let $\dbR^{n\times m}$ be the space of $n\times m$ real matrices equipped with the Frobenius inner product
$$ \lan M,N\ran=\tr(M^\top N), \q M,N\in\dbR^{n\times m}, $$
where $M^\top$ denotes the transpose of $M$ and $\tr(M^\top N)$ is the trace of $M^\top N$.
The norm induced by the Frobenius inner product is denoted by $|\cd|$.
For a subset $\dbH$ of $\dbR^{n\times m}$, we denote by $C([0,T];\dbH)$ the space of continuous functions
from $[0,T]$ into $\dbH$, and by $L^\i(0,T;\dbH)$ the space of Lebesgue measurable, essentially bounded
functions from $[0,T]$ into $\dbH$.
Let $\dbS^n$ be the subspace of $\dbR^{n\times n}$ consisting of symmetric matrices and $\dbS^n_+$ the
subset of $\dbS^n$ consisting of positive definite matrices.
For $\dbS^n$-valued functions $M(\cd)$ and $N(\cd)$, we write $M(\cd)\ges N(\cd)$ (respectively, $M(\cd)>N(\cd)$)
if $M(\cd)-N(\cd)$ is positive semidefinite (respectively, positive definite) almost everywhere with respect
to the Lebesgue measure.
The identity matrix of size $n$ is denoted by $I_n$, and a vector always refers to a column vector if not specified.
Also, recall that $W=\{W(t);\,t\ges0\}$ is a standard one-dimensional Brownian motion, $\dbF=\{\sF_t\}_{t\ges0}$
is the usual augmentation of the natural filtration generated by $W$, and that $\sU[0,T]$ is the space of
$\dbR^m$-valued, $\dbF$-progressively measurable, square-integrable processes over $[0,T]$.

\ms

For the purpose of later presentation, we recall some results of time-variant stochastic LQ problem in finite horizon. Consider the state equation
\bel{state:SLQ}\left\{\2n\ba{ll}
\ns\ds d\BX(t)=[\BA(t)\BX(t)+\BB(t)\Bu(t)+\Bb(t)]dt+ [\BC(t)\BX(t)+\BD(t)\Bu(t)+\BBsi(t)]dW(t), \\
\ns\ds\qq\qq\qq\qq\qq\qq\qq\qq\qq\qq\qq\qq\qq\qq t\in[0,T],\\
\ns\ds\BX(0)=x,\ea\right.\ee
with the cost functional
\bel{cost:SLQ}\ba{ll}
\ns\ds\BJ(x;\Bu(\cd))={1\over2}\dbE\Big\{\lan\BG\BX(T),\BX(T)\ran+2\lan \Bg,\BX(T)\ran\\
\ns\ds\qq\qq+\int_0^T\[
\Blan\begin{pmatrix}\BQ(t)&\BS(t)^\top\\ \BS(t)&\BR(t)\end{pmatrix}\!
     \begin{pmatrix}\BX(t) \\ \Bu(t)\end{pmatrix}\!,
     \begin{pmatrix}\BX(t) \\ \Bu(t) \end{pmatrix}\Bran
+2\Blan\begin{pmatrix}\Bq(t) \\ \Br(t) \end{pmatrix}\!,
       \begin{pmatrix}\BX(t) \\ \Bu(t) \end{pmatrix}\Bran\] dt\Big\},\ea\ee
where in \rf{state:SLQ}, the coefficients satisfy
$$\BA(\cd),\BC(\cd)\1n\in\1n L^\i(0,T;\dbR^{n\times n}),
\q\BB(\cd),\BD(\cd)\1n\in\1n L^\i(0,T;\dbR^{n\times m}),
\q\Bb(\cd),\BBsi(\cd)\1n\in\1n L^\i(0,T;\dbR^n),$$
and in \rf{cost:SLQ}, the weighting coefficients satisfy
$$\ba{ll}
\ns\ds\BG\in\dbS^n,\q\BQ(\cd)\in L^\i(0,T;\dbS^n),\q\BS(\cd)\in L^\i(0,T;\dbR^{m\times n}),\q\BR(\cd)\in L^\i(0,T;\dbS^m),\\
\ns\ds\Bg\in\dbR^n,\q\Bq(\cd)\in L^\i(0,T;\dbR^n),\q\Br(\cd)\in L^\i(0,T;\dbR^m).\ea$$
The standard stochastic LQ optimal control problem on $[0,T]$ can be stated as follows.

\begin{taggedthm}{Problem (SLQ$_{[0,T]}$).}
For a given initial state $x\in\dbR^n$, find a control $\bar\Bu(\cd)\in\sU[0,T]$ such that
\bel{SLQ:def-v*}\BJ(x;\bar\Bu(\cd))=\inf_{\Bu(\cd)\in\sU[0,T]}
\BJ(x;\Bu(\cd))\equiv\BV(x).\ee
\end{taggedthm}

The process $\bar\Bu(\cd)$ (if exists) in \rf{SLQ:def-v*} is called an ({\it open-loop}) {\it optimal control} for the initial state $x$, and $\BV(x)$ is called the {\it value} of Problem (SLQ$_{[0,T]}$) at $x$.

\ms

The following lemma summarizes a few results for Problem (SLQ).
For proofs, the reader is referred to the book \cite{Sun-Yong2020} by Sun and Yong;
see also \cite{Sun-Li-Yong2016}.

\begin{lemma}\label{lmm:opti-sys}
Suppose that for some constant $\d>0$,
\bel{BR>0}\BR(t)\ges\d I,\qq\BQ(t)-\BS(t)^\top\BR(t)^{-1}\BS(t)\ges0.\ee
Then, the following hold:

\ms

{\rm(i)} For every initial state $x\in\dbR^n$, Problem {\rm(SLQ)$_{[0,T]}$} has a unique open-loop optimal control.

\ms

{\rm(ii)} A pair $(\bar\BX(\cd),\bar\Bu(\cd))$ is an open-loop optimal pair of Problem {\rm(SLQ)$_{_T}$} the initial state $x$ if and only if there exists a pair $(\bar\BY(\cd),\bar\BZ(\cd))$ of adapted processes such that
\bel{SLQ:opti-sys}\left\{\2n\ba{ll}
\ns\ds d\bar\BX=\big[\BA(t)\bar\BX(t)+\BB(t)\bar\Bu(t)+\Bb(t)\big]dt
+\big[\BC(t)\bar\BX(t)+\BD(t)\bar\Bu(t)+\BBsi(t)\big]dW, \\
\ns\ds d\bar\BY(t)=-\big[\BA(t)^\top\bar\BY(t)+\BC(t)^\top\bar\BZ(t)+\BQ(t)
\bar\BX(t)+\BS(t)^\top\bar\Bu(t)+\Bq(t)\big]dt+\bar\BZ(t)dW(t),\\
\ns\ds\bar\BX(0)=x,\qq\bar\BY(T)=\BG\bar\BX(T)+\Bg,\ea\right.\ee
and the following condition holds:
\bel{SLQ:sta-cdtn}\BB(t)^\top\bar\BY(t)+\BD(t)^\top\bar\BZ(t)+\BS(t)
\bar\BX(t)+\BR(t)\bar\Bu(t)+\Br(t)=0,\q\ae\; t\in[0,T],~\as\ee

\ms

{\rm(iii)} The Riccati differential equation
\bel{Ric:SLQ}\left\{\2n\ba{ll}
\ns\ds\dot\BP(t)+\BP(t)\BA(t)+\BA(t)^\top\BP(t)+\BC(t)^\top \BP(t)\BC(t)+\BQ(t)\\
\ns\ds\qq-\big[\BP(t)\BB(t)+\BC(t)^\top\BP(t)\BD(t)+\BS(t)^\top\big]
\big[\BR(t)+\BD(t)^\top\BP(t)\BD(t)\big]^{-1}\\
\ns\ds\qq\qq\cd\big[\BB(t)^\top \BP(t)+\BD(t)^\top\BP(t)\BC(t)+\BS(t)\big]=0, \\
\ns\ds\BP(T)=\BG\ea\right.\ee
admits a unique positive semidefinite solution $\BP(\cd)\in C([0,T];\dbS^n)$. In particular, if
$$\BQ(t)-\BS(t)^\top\BR(t)^{-1}\BS(t)>0,$$
then $\BP(t)>0$ for all $t\in[0,T)$.

\ms

{\rm(iv)} The unique open-loop optimal control $\bar\Bu(\cd)$ for the initial state $x$ is given by
$$\bar\Bu(t)=\BTh(t)\bar\BX(t)-\big[\BR(t)+\BD(t)^\top \BP(t)\BD(t)\big]^{-1}\big[\BB(t)^\top\BBf(t)+\BD(t)^\top \BP(t)\BBsi(t)+\Br(t)\big],$$
where
$$\BTh(t)=-\big[\BR(t)+\BD(t)^\top \BP(t)\BD(t)\big]^{-1}\big[\BB(t)^\top\BP(t)+\BD(t)^\top\BP(t)\BC(t)
+\BS(t)\big],$$
and $\BBf(\cd)$ is the solution to the terminal value problem of the ordinary differential equation (ODE, for short)
\bel{SLQ:ODE-f}\left\{\2n\ba{ll}
\ns\ds\dot\BBf(t)+\big[\BA(t)+\BB(t)\BTh(t)\big]^\top\BBf(t)
+\big[\BC(t)+\BD(t)\BTh(t)\big]^\top\BP(t)\BBsi(t)\\
\ns\ds\qq\qq\qq\qq+\BTh(t)^\top\Br(t)+\BP(t)\Bb(t)+\Bq(t)=0,\qq t\in[0,T],\\
\ns\ds\BBf(T)=\Bg.\ea\right.\ee

\ms

{\rm(v)} The value function is given by
$$\ba{ll}
\ds\BV(x)={1\over2}\lan\BP(0)x,x\ran+\lan\BBf(0),x\ran+{1\over2}\int_0^T\(\lan \BP(t)\BBsi(t),\BBsi(t)\ran+2\lan\BBf(t),\Bb(t)\ran\\
\ns\ds\qq\qq-\big|\big[\BR(t)+\BD(t)^\top \BP(t)\BD(t)\big]^{-{1\over2}}\big[\BB(t)^\top\BBf(t)+\BD(t)^\top \BP(t)\BBsi(t)+\Br(t)\big]\big|^2\)dt.\ea$$

\end{lemma}

Now, we return to our Problem (SLQ)$_{_T}$. Let us make a reduction under the strong standard conditions \rf{R>0}. Set
\bel{u=v}u(t)=v(t)-R^{-1}SX(t),\qq\in[0,T],\ee
with $v(\cd)\in\sU[0,T]$. We observe the following ($t$ is suppressed)
$$\ba{ll}
\ns\ds\lan QX,X\ran+2\lan SX,u\ran+\lan Ru,u\ran+2\lan q,X\ran+2\lan r,u\ran\\
\ns\ds=\1n\lan QX,X\ran\1n+\1n2\lan SX,v\1n-\1n R^{-1}SX\1n-
\1n R^{-1}r\ran\1n+\1n\lan R(v\1n-\1n R^{-1}SX\1n-\1n R^{-1}r),v\1n-\1n R^{-1}SX\1n-\1n R^{-1}r\ran\\
\ns\ds\qq+2\lan q,X\ran+2\lan r,v-R^{-1}SX-R^{-1}r\ran\\
\ns\ds=\1n\lan(Q\1n-\1n S^\top R^{-1}S)X,X\ran\1n+\1n\lan R(v\1n-\1n R^{-1}r),v\1n-\1n R^{-1}r\ran\1n+\1n2\lan q\1n-\1n S^\top R^{-1}r,X\ran\1n+\1n2\lan r,v\1n-\1n R^{-1}r\ran\\
\ns\ds=\lan\h QX,X\ran+\lan Rv,v\ran+2\lan q-S^\top R^{-1}r,X\ran-\lan R^{-1}r,r\ran\\
\ns\ds\equiv\lan\h QX,X\ran+\lan Rv,v\ran+2\lan\h q,X\ran-\f_0,\ea$$
where
$$\h Q=Q-S^\top R^{-1}S,\qq\h q=q-S^\top R^{-1}r,\qq\f_0=r^\top R^{-1}r.$$
Also, under \rf{u=v}, one has
$$\ba{ll}
\ns\ds AX+Bu+b=(A-R^{-1}S)X+Bv+b-BR^{-1}r\equiv\h AX+Bv+\h b,\\
\ns\ds CX+Du+\si=(C-R^{-1}S)X+Dv+\si-DR^{-1}r\equiv\h CX+Dv+\h\si,\ea$$
where
$$\ba{ll}
\ns\ds\h A=A-R^{-1}S,\qq\h b=b-BR^{-1}r,\\
\ns\ds\h C=C-R^{-1}S,\qq\h\si=\si-DR^{-1}r.\ea$$
From the above reduction, we end up with the following state
equation:
\bel{state2}\left\{\2n\ba{ll}
\ds dX(t)=\(\h AX(t)+Bv(t)+\h b\)dt+\(\h CX(t)+Dv(t)+\h\si\)dW(t),\\
\ds X(0)=x,\ea\right.\ee
with the cost functional
\bel{cost2}J_{\scT}(x;v(\cd))={1\over2}\dbE\int_0^T\big[\lan\h QX(t),X(t)\ran+\lan Rv(t),v(t)\ran+2\lan\h q,X(t)\ran-\f_0\big]dt.\ee
It is clear that the (open-loop) optimal pair of the LQ problem
associated with \rf{state2}--\rf{cost2} is the same as that of
the LQ problem associated with the state equation \rf{state2} and the cost functional
$$\h J_{\scT}(x;v(\cd))={1\over2}\dbE\2n\int_0^T\3n\big[\lan\h QX(t),X(t)\ran+\lan Rv(t),v(t)\ran+2\lan\h q,X(t)\ran\big]dt\equiv J_{\scT}(x;v(\cd))+{\f_0T\over2}.$$
Because of the above reduction, one sees that it suffices to consider Problem (SLQ)$_{_T}$ for the state equation \rf{state} with the following cost functional
\bel{cost3}J_{\scT}(x;u(\cd))={1\over2}\dbE\int_0^T\big[\lan QX(t),X(t)\ran+\lan Ru(t),u(t)\ran+2\lan q,X(t)\ran\big]dt.\ee
In the case that $b=\si=q=0$, we denote the corresponding by Problem (SLQ)$^0_{\scT}$, and call it a homogenous LQ problem on $[0,T]$. The value function of Problem (SLQ)$^0_{\scT}$ is denoted by $V_{\scT}^0(x)$.

\section{Stabilizability and the Static Optimization Problem}\label{Sec:Stability}

In this section, we are going to make some further preparations.

\ms

Let us denote by $[A,C]$ the following linear homogeneous uncontrolled SDE:
\bel{[A,C]}dX(t)=AX(t)dt+CX(t)dW(t),\q t\ges0.\ee
For any $x\in\dbR^n$, there exists a unique solution $X(\cd)\equiv X(\cd\,;x)$ of the above satisfying $X(0;x)=x$.
We recall the following classical notion.

\begin{definition} System $[A,C]$ is said to be

\ms

(i) {\it $L^2$-stable} if
\bel{L^2}\dbE\int_0^\infty|X(t;x)|^2dt<\infty,\qq\forall x\in\dbR^n.\ee

(ii) {\it mean-square exponentially stable} if there exists a $\b>0$ such that
\bel{-bt}\sup_{t\in[0,\infty)}e^{\b t}|X(t;x)|<\infty,\qq\forall x\in\dbR^n.\ee

\end{definition}

To characterize the above notions, we let $\F(\cd)$ be the solution to the matrix SDE
\bel{SDE:Phi0}\left\{\2n\ba{ll}
\ds d\F(t)= A\F(t)dt+C\F(t)dW(t),\q t\ges0,\\
\ds\F(0)=I_n.\ea\right.\ee
We have the following result.

\bl{equivalence} \sl The following are equivalent:

\ms

{\rm(i)} The system $[A,C]$ is mean-square exponentially stable.

\ms

{\rm(ii)} There exist constants $\a,\b>0$ such that
\bel{e-stable}\dbE|\F(t)|^2\les\a e^{-\b t},\q\forall t\ges0.\ee

\ms

{\rm(iii)} It holds that
\bel{L2-stable}\dbE\int_0^\i|\F(t)|^2dt<\i.\ee

\ms

{\rm(iv)} The system $[A,C]$ is $L^2$-stable.

\ms

{\rm(v)} There exists a $P\in\dbS^n_+$ such that
\bel{PA+AP+CPC<0}PA+A^\top P+C^\top PC<0.\ee

\el

The proof is straightforward (see \cite{Ait-Zhou2000}, \cite{Huang-Li-Yong2015}, or the book \cite{Sun-Yong2020}).

\begin{corollary}\label{crllry:e-stable}
If the system $[A,C]$ is $L^2$-stable, then
$$ |e^{At}| \les \sqrt{\a }e^{-(\b t/2)}, \q\forall t\ges0,$$
where $\a$ and $\b$ are as in \rf{e-stable}.
\end{corollary}

\begin{proof} It is easy to prove by noting that $e^{At}=\dbE[\F(t)]$.
\end{proof}

Now we denote by $[A,C;B,D]$ the following controlled (homogeneous) linear system:
\bel{[A,C;B,D]}\left\{\2n\ba{ll}
\ds dX(t)=[AX(t)+Bu(t)]dt+[CX(t)+Du(t)]dW(t),\q t\ges0,\\
\ns\ds X(0)=x,\ea\right.\ee
where $u(\cd)$ is taken from the following set of {\it admissible controls}
$$\sU]0,\infty)\1n=\1n\Big\{u\1n:\1n[0,\infty)\1n\times\1n\Om\to\dbR^m\bigm|u(\cd)\hb{ is $\dbF$-progressivey measurable, }
\dbE\int_0^\infty\3n|u(t)|^2dt\1n<\1n\infty\Big\}.$$
We recall the following notion (see \cite{Huang-Li-Yong2015} and \cite{Sun-Yong2020}).

\begin{definition}\label{def:stabilizable}
The system $[A,C;B,D]$ is said to be {\it $L^2$-stabilizable} if there exists a matrix $\Th\in\dbR^{m\times n}$
such that the (closed-loop) system $[A+B\Th,C+D\Th]$ is $L^2$-stable.
In this case, $\Th$ is called a {\it stabilizer} of $[A,C;B,D]$.
\end{definition}

We point out that the $L^2$-stabilizability is a weaker condition than the {\it null controllability}, meaning that for any initial state $x$ there exists a control to steer the system state from $x$ to $0$ in a finite time interval. In fact, according to \cite[Theorem 3.2]{Sun-Yong2018}, the system $[A,C;B,D]$ is $L^2$-stabilizable
if and only if for any initial state $x$, there exists an $\dbR^m$-valued process $u(\cd)\in\sU[0,\infty)$ such that the solution $X(\cd\,;x,u(\cd))$ of \rf{[A,C;B,D]} is also square-integrable over $[0,\i)$. Whereas, the system $[A,C;B,D]$ is null controllable, then, for any $x\in\dbR^n$, we can find a $v(\cd)\in\sU[0,T]$ for some $T>0$, such that $X(T;x,v(\cd))=0$. Thus, with
$$u(t)\deq\left\{\begin{aligned}
& v(t),  && t\in[0,T], \\
& 0,     && t>T,
\end{aligned}\right.$$
the corresponding solution $X(\cd)\equiv X(\cd\,;x,u(\cd))$ satisfies $\dbE\int_0^\i|X(t)|^2 dt<\i$. This shows that the null controllability is stronger than the $L^2$-stabilizability. Controllability for linear ODEs is very standard in control theory, see \cite{Kalman1960}. For linear SDEs, the situation is much more complicated, see \cite{Peng1994, Liu-Peng2010, Wang-Yang-Yong-Yu2017} for some known results.

\ms

We now introduce the following hypotheses.

\begin{taggedassumption}{(H1)}\label{TP:H1}
The system $[A,C;B,D]$ is $L^2$-stabilizable.
\end{taggedassumption}

\begin{taggedassumption}{(H2)}\label{TP:H2}
The weighting matrices $Q\in\dbS^n_+$ and $R\in\dbS^m_+$.
\end{taggedassumption}

Under the above \ref{TP:H1}--\ref{TP:H2}, we may consider the  state equation \rf{[A,C;B,D]} with the cost functional
\bel{cost0}J^0_\infty(x;u(\cd))={1\over2}\dbE\int_0^\infty\(\lan QX(t),X(t)\ran+\lan Ru(t),u(t)\ran\)dt.\ee
We could formulate the following homogeneous LQ problem in the infinite horizon $[0,\infty)$.

\ms

\bf Problem (SLQ)$_\infty^0$. \rm For each $x\in\dbR^n$, find $\bar u(\cd)\in\sU[0,\infty)$ such that
\bel{inf}J^0_\infty(x;\bar u(\cd))=\inf_{u(\cd)\in\sU[0,\infty)}J^0_\infty(x;u(\cd))\equiv V^0_\infty(x).\ee

The following collects the relevant results of Problem (SLQ)$_\infty^0$. See \cite{Sun-Yong2020} for a proof.

\bp{LQ infinite} \sl Let {\rm\ref{TP:H1}--\ref{TP:H2}} hold. Then for each $x\in\dbR^n$, Problem (LQ)$_\infty^0$ admits a unique open-loop optimal control $\bar u(\cd)$. Moreover, the following ARE
\bel{ARE2}PA+A^\top P+C^\top PC+Q-(PB+C^\top PD)(R+D^\top PD)^{-1}(B^\top P+D^\top PC)=0,\ee
admits a unique solution $P\in\dbS^n_+$ such that the open-loop optimal control $\bar u(\cd)$ admits the following {\it closed-loop representation}:
\bel{bar u}\bar u(t)=-\Th(t)\bar X(t),\qq t\in[0,\infty),\ee
where
\bel{Th1}\Th=-(R+D^\top PD)^{-1}(B^\top P+D^\top PC)\ee
is a stabilizer of $[A,C;B,D]$, and the value function has the following quadratic form:
\bel{V^0}V^0_\infty(x)=\lan Px,x\ran,\qq\forall x\in\dbR^n.\ee

\ep

\ms

In the above case, $P$ is referred to as a {\it stabilizing solution} of the ARE with respect to the system $[A,C;B,D]$. Also, by a direct comparison, making use of \ref{TP:H2}, we see that
\bel{V<V}V_{\scT}^0(x)\les V_\infty^0(x),\qq\forall x\in\dbR^n.\ee

\ms

Now, we define
\bel{V,F}\ba{ll}
\ds\sV=\{(x,u)\in\dbR^n\times\dbR^m\bigm|Ax+Bu+b=0\},\\
\ns\ds F(x,u)=\lan Qx,x\ran+\lan Ru,u\ran+2\lan q,x\ran+\lan P(Cx+Du+\si),Cx+Du+\si\ran,\ea\ee
and introduce the following static optimization problem.

\ms

\bf Problem (O). \rm Find $(x^*,u^*)\in\sV$ such that
\bel{min F}F(x^*,u^*)=\min_{(x,u)\in\sV}F(x,u).\ee

\ms

For the above static optimization problem, we have the following result.

\bp{static*} \sl Let {\rm\ref{TP:H1}--\ref{TP:H2}} hold. Then Problem {\rm(O)} admits a unique solution $(x^*,u^*)\in\sV$ which, together with a {\it Lagrange multiplier} $\l^*\in\dbR^n$, is characterized by the following system of linear equations:
\bel{L-conditions}\left\{\2n\ba{ll}
\ds Qx^*+A^\top\l^*+C^\top P(Cx^*+Du^*+\si)+q=0,\\
\ds Ru^*+B^\top\l^*+D^\top P(Cx^*+Du^*+\si)=0.\ea\right.\ee

\ep

\begin{proof} \rm Since $[A,C;B,C]$ is stabilizable, so must be $[A;B]\equiv[A,0;B,0]$. In fact, if $\Th$ is a stabilizer of $[A,C;B,D]$, then there exists a $P\in\dbS^n_+$ such that
$$P(A+B\Th)+(A+B\Th)^\top P+(C+D\Th)^\top P(C+D\Th)<0,$$
which leads
$$P(A+B\Th)+(A+B\Th)^\top P<0.$$
Thus, $[A;B]$ is stabilizable which is equivalent to that the matrix $(A-\l I,B)$ is of full rank for any $\l\in\dbC$ with $\hb{Re}\,\l\ges0$ (\cite{Li-Yong-Zhou2010}). Consequently, by taking $\l=0$, one sees that the matrix $(A,B)$ has rank $n$. Then the feasible set $\sV$ of Problem (O) is a non-empty closed convex set. Further, since $P>0$,
$$F(x,u)\ges\lan Qx,x\ran+\lan Ru,u\ran+2\lan q,x\ran$$
is coercive on $\dbR^n\times\dbR^m$ due to \ref{TP:H2}. Therefore, Problem (O) admits a unique solution.

\ms

Next, let
$$G(x,u)=Ax+Bu+b=(A,B)\begin{pmatrix}x\\ u\end{pmatrix}.$$
Then
$$G_{(x,u)}(x,u)=(A,B),$$
which is of full rank. Hence, the equality constraint is regular. Consequently, the optimal solution $(x^*,u^*)$ can be obtained by the Lagrange multiplier method (\cite{Yong2018}). Now, we form Lagrange function:
$$L(x,u,\l)=F(x,u)+2\l^\top(Ax+Bu+b).$$
Suppose $(x^*,u^*)$ is the unique optimal solution of the above problem. Then
\bel{L-conditions2}\ba{ll}
\ns\ds0={1\over2}L_x(x^*,u^*,\l^*)^\top=(Q+C^\top PC)x^*+C^\top PDu^*+q+C^\top P\si+A^\top\l^*,\\
\ns\ds0={1\over2}L_u(x^*,u^*,\l^*)^\top=(R+D^\top PD)u^*+D^\top PCx^*+D^\top P\si+B^\top\l^*.\ea\ee
This leads to \rf{L-conditions}. Further, we may write the above as follows:
$$\begin{pmatrix}Q+C^\top PC&C^\top PD\\ D^\top PC&R+D^\top PD\end{pmatrix}\begin{pmatrix}x^*\\ u^*\end{pmatrix}=\begin{pmatrix}C^\top\\ D^\top\end{pmatrix}P\si+\begin{pmatrix}A^\top\\ B^\top\end{pmatrix}\l^*.$$
Since the coefficient matrix is invertible, there exists a unique solution $(x^*,u^*)$ and
$$\begin{pmatrix}x^*\\u^*\end{pmatrix}=\begin{pmatrix}Q+C^\top PC&C^\top PD\\ D^\top PC&R+D^\top PD\end{pmatrix}^{-1}\[\begin{pmatrix}C^\top\\ D^\top\end{pmatrix}P\si+\begin{pmatrix}A^\top\\ B^\top\end{pmatrix}\l^*\].$$
By the equality constraint, one has
$$-b=Ax^*+Bu^*=(A,B)\begin{pmatrix}Q+C^\top PC&C^\top PD\\ D^\top PC&R+D^\top PD\end{pmatrix}^{-1}\[\begin{pmatrix}C^\top\\ D^\top\end{pmatrix}P\si+\begin{pmatrix}A^\top\\ B^\top\end{pmatrix}\l^*\].$$
Since $(A,B)\in\dbR^{n\times(n+m)}$ has rank $n$, we see that
$\l^*$ is uniquely determined by the following:
$$\ba{ll}
\ns\ds\l^*=-\[(A,B)\begin{pmatrix}Q+C^\top PC&C^\top PD\\ D^\top PC&R+D^\top PD\end{pmatrix}^{-1}\begin{pmatrix}A^\top\\ B^\top\end{pmatrix}\]^{-1}\\
\ns\ds\qq\qq\cd\[b+(A,B)\begin{pmatrix}Q+C^\top PC&C^\top PD\\ D^\top PC&R+D^\top PD\end{pmatrix}^{-1}\begin{pmatrix}C^\top\\ D^\top\end{pmatrix}P\si\].\ea$$
Hence, $(x^*,u^*)$ is uniquely determined by \rf{L-conditions}. \end{proof}

\ms

Now, we make some simple observation on problems \rf{Fake} and Problem (O) (or \rf{static}). Let
$$\sV_0=\{(x,u)\in\dbR^n\times\dbR^m\bigm|Ax+Bu+b=0,~Cx+Du+\si=0\}.$$
Then, it is clear that $\sV_0$ is a subset of $\sV$, and $\sV_0$ could even be empty. Moreover, if $(x^*,u^*)\in\sV$ is an optimal solution of Problem (O) and $(x^*,u^*)\in\sV_0$, then it is an optimal solution to problem \rf{Fake}. However, if Problem (O) has a unique solution $(x^*,u^*)\notin\sV_0$ (in particular, if $\sV_0=\varnothing$), then problems \rf{Fake} and Problem (O) are totally different. We present two illustrative examples below.

\bex{} \rm Let $n=m=1$ and consider the following state equation:
$$\left\{\2n\ba{ll}
\ds dX(t)=u(t)dt+X(t)dW(t),\qq t\ges0,\\
\ns\ds X(0)=x,\ea\right.$$
with cost functional
$$J_{\scT}(u(\cd))={1\over2}\dbE\int_0^T\(2|X(t)|^2+|u(t)|^2+4X(t)\)dt.$$
We see that in this case,
$$A=D=0,\q B=C=1,\q b=\si=0,\q Q=2,\q R=1,\q q=2.$$
Then the ARE reads:
$$0=P+1-P^2=-(P-2)(P+1).$$
Hence the positive definite solution is $P=2$. Thus, Problem (O) is equivalent to
$$\left\{\2n\ba{ll}
\ns\ds\hb{Minimize}\q 4x^2+u^2+4x,\\
\ns\ds\hb{Subject to}\q u=0.\ea\right.$$
It is straightforward that the solution is given by
$$x^*=-{1\over2},~u^*=0.$$
Whereas, \rf{Fake} reads
$$\left\{\2n\ba{ll}
\ns\ds\hb{Minimize}\q2x^2+u^2+4x,\\
\ns\ds\hb{Subject to }~x=0,\q u=0,\ea\right.$$
whose solution is trivially given by
$$\bar x^*=0,~\bar u^*=0.$$
Hence, the solutions to these two problems are different.

\ex

\bex{} \rm Let $n=m=1$ and consider the following state equation:
$$\left\{\2n\ba{ll}
\ds dX(t)=\big[X(t)+u(t)+1\big]dt+\big[X(t)+u(t)\big]dW(t),\qq t\ges0,\\
\ns\ds X(0)=x,\ea\right.$$
with cost functional
$$J_{\scT}(u(\cd))={1\over2}\dbE\int_0^T\(|X(t)|^2+|u(t)|^2\)dt.$$
We see that in this case,
$$A=B=C=D=Q=R=b=1,\q\si=q=0.$$
Then
$$\sV_0=\{(x,u)\bigm|x+u+1=0,x+u=0\}=\varnothing.$$
On the other hand, the ARE reads:
$$3P+1-{4P^2\over1+P}=0,$$
which is equivalent to
$$P^2-4P-1=0.$$
Thus, the positive solution is $P=2+\sqrt5$. Hence,
$$F(x,u)=x^2+u^2+(2+\sqrt5)|x+u|^2.$$
Consequently, Problem (O) is well-formulated and it admits a unique optimal solution.

\ex

From the above examples, together with our main result on the turnpike property of the Problem (SLQ)$_{_T}$ which will be presented a little later, we see that \rf{Fake} is not a suitable problem to be considered and the correct on is Problem (O).

\section{Convergence of the Riccati Equation}\label{Sec:Convergence}


For notational simplicity, we define, for each $P\in\dbS^n$:
\begin{equation}\label{cQSR(P)}\left\{\begin{aligned}
\cQ(P) &= PA+A^\top P+C^\top PC+Q,\\
\cS(P) &= B^\top P+D^\top PC,\\
\cR(P) &= R+D^\top PD, \\
\cK(P) &= -\cR(P)^{-1}\cS(P), \text{ provided $\cR(P)$ is invertible.}
\end{aligned}\right.\end{equation}
Then ARE \rf{ARE2} can be written as follows:
\bel{ARE3}\cQ(P)-\cS(P)^\top\cR(P)^{-1}\cS(P)=0.\ee
By Proposition \ref{LQ infinite}, under \ref{TP:H1}--\ref{TP:H2}, the above \rf{ARE3} admits a stabilizing solution $P\in\dbS^n_+$
(with respect to $[A,C;B,D]$), and \rf{V^0} holds.

\ms

According to \autoref{lmm:opti-sys}, we know that under \ref{TP:H1}--\ref{TP:H2}, for each $T>0$, Problem (SLQ)$_{\scT}$
admits a unique optimal control for every initial state $x$. Moreover, the following conclusions hold:

\ms

(i) The optimal pair $(\bar X_{\scT}(\cd),\bar u_{\scT}(\cd))$, together with a pair of adapted processes $(\bar Y_{_T}(\cd)$, $\bar Z_{_T}(\cd))$, satisfies the following optimality system:
\bel{os:ProbT}\left\{\2n\ba{ll}
\ds d\bar X_{\scT}(t)=[A\bar X_{\scT}(t)+B\bar u_{\scT}(t)+b]dt+[C\bar X_{\scT}(t)+D\bar u_{\scT}(t)+\si]dW(t),\\
\ns\ds d\bar Y_{\scT}(t)=-\lt[A^\top\bar Y_{\scT}(t)+C^\top\bar Z_{\scT}(t)+Q\bar X_{\scT}(t)+q\rt]dt+\bar Z_{\scT}(t)dW(t),\\
\ns\ds\bar X_{\scT}(0)=x,\q\bar Y_{\scT}(T)=0,\\
\ns\ds B^\top\bar Y_{\scT}(t)+D^\top\bar Z_{\scT}(t)+R\bar u_{\scT}(t)= 0, \q\ae~t\in[0,T],~\as,\ea\right.\ee

\ms

(ii) The differential Riccati equation
\bel{Ric:PT}\left\{\2n\ba{ll}
\ds\dot P_{\scT}(t)+\cQ(P_{\scT}(t))-\cS(P_{\scT}(t))^\top \cR(P_{\scT}(t))^{-1}\cS(P_{\scT}(t))=0,\q t\in[0,T],\\
\ns\ds P_{\scT}(T)=0\ea\right.\ee
admits a unique solution $P_{\scT}(\cd)\in C([0,T];\dbS^n)$ satisfying $P_{\scT}(t)>0$ for all $0\les t<T$, and the (open-loop) optimal control has the following state feedback representation:
\bel{bar u00}\bar u(t)=-\cK(P_{\scT}(t))\bar X(t)-\cR(P_{\scT}(t))^{-1}\big[B^\top\f+D^\top P_{\scT}(t)\si\big],\ee
where $\f(\cd)$ is the solution to the following:
\bel{SLQ:ODE-f}\left\{\2n\ba{ll}
\ns\ds\dot\f(t)\1n+\1n\big[A\1n-\1n B\cK(P_{\scT}(t))\big]^\top\f(t)
\1n+\1n\big[C\1n-\1n D\cK(P_{\scT}(t))\big]^\top P_{\scT}(t)\si(t)\1n+\1n P_{\scT}(t)b\1n+\1n q\1n=\1n0,\\
\ns\ds\qq\qq\qq\qq\qq\qq\qq\qq\qq\qq\qq\qq\qq t\in[0,T],\\
\ns\ds\f(T)=0.\ea\right.\ee

\ms

{\rm(v)} The value function is given by
$$\ba{ll}
\ds V_{\scT}(x)={1\over2}\lan P_{\scT}(0)x,x\ran+\lan\f(0),x\ran+{1\over2}\int_0^T\(\lan P_{\scT}(t)\si,\si\ran+2\lan\f(t),b\ran\\
\ns\ds\qq\qq-\big|\cR(P_{\scT}(t))^{-{1\over2}}\big[B^\top\f(t)+D^\top P_{\scT}(t)\si\big]\big|^2\)dt.\ea$$

\ss

Note that by setting $b=\si=q=0$, we have the corresponding result for the (homogeneous) Problem (SLQ)$_{\scT}^0$. For such a case, $\f(\cd)=0$, and in particular, also taking into account of \rf{V^0}--\rf{V<V},
\bel{V^0_T}V^0_{\scT}(x)={1\over2}\lan P_{\scT}(0)x,x\ran\les{1\over2}\lan Px,x\ran=V^0_\infty(x),\qq x\in\dbR^n.\ee

\ms

We now look at the convergence property of $P_{\scT}(\cd)$ as $T\to\i$, which plays an essential role in the turnpike property of Problem (SLQ)$_{\scT}$. For this, let us present the following result first.

\begin{proposition}\label{prop:Si-increase}
Let {\rm\ref{TP:H1}--\ref{TP:H2}} hold. Then the equation
\begin{equation}\label{Ric:Sigma}\left\{\begin{aligned}
& \dot{\Si}(t)-\cQ(\Si(t))+\cS(\Si(t))^\top\cR(\Si(t))^{-1}\cS(\Si(t))=0, \q t\in[0,\i),\\
& \Si(0) =0
\end{aligned}\right.\end{equation}
admits a unique solution $\Si(\cd)\in C([0,\i);\dbS^n)$ satisfying
\bel{<P}0<\Si(s)\les\Si(t)\les P, \qq\forall\, 0<s<t<\i,\ee
with $P\in\dbS_+^n$ being the stabilizing solution of \rf{ARE3} with respect to $[A,C;B,D]$. Moreover, $\Si(T-t)=P_{\scT}(t)$ for every $0\les t\les T$.

\end{proposition}

\begin{proof} For fixed but arbitrary $0<T_1<T_2<\i$, we define
$$\ba{ll}
\ds\Si_1(t)\deq P_{_{\sc T_1}}(T_1-t),\qq0\les t\les T_1,\\
\ds\Si_2(t)\deq P_{_{\sc T_2}}(T_2-t),\qq 0\les t\les T_2.\ea$$
On the interval $[0,T_1]$, both $\Si_1$ and $\Si_2$ solve the same equation
$$\left\{\2n\ba{ll}
\ds\dot\Si(t)-\cQ(\Si(t))+\cS(\Si(t))^\top\cR(\Si(t))^{-1}\cS(\Si(t))=0, \\
\ds\Si(0)=0.\ea\right.$$
By the uniqueness, we must have
\bel{Si1=Si2}\Si_1(t)=\Si_2(t),\qq\forall t\in[0,T_1].\ee
Then the function $\Si:[0,\i)\to\dbS^n$ defined by
$$ \Si(t) \deq P_{\scT}(T-t) $$
is independent of the choice of $T\ges t$ and is a solution of \rf{Ric:Sigma}. We now claim that
$$0<\Si(s)\les\Si(t)\les P,\qq\forall\,0<s<t<\i.$$
To see this, we note that for any $0\les T_1<T_2<\infty$, by \ref{TP:H2}, we have
$$\lan P_{\scT_1}(0)x,x\ran=2V^0_{\scT_1}(x)\les2V^0_{\scT_2}(x)=
\lan P_{\scT_2}(0)x,x\ran,\qq\forall x\in\dbR^n.$$
Thus,
$$\Si(T_1)=P_{\scT_1}(0)\les P_{\scT_2}(0)=\Si(T_2).$$
Finally, by \rf{V^0_T}, we obtain our conclusion.
\end{proof}

The above result leads to the following convergence.

\begin{proposition}\label{prop:sol-ARE} Let {\rm\ref{TP:H1}--\ref{TP:H2}} hold and $\Si(\cd)$ be the solution to the ODE \rf{Ric:Sigma}. Then the limit $\ds P_\infty\deq \lim_{t\to\i}\Si(t)$ exists, which is the stabilizing solution of the ARE \rf{ARE3} with respect to $[A,C;B,D]$, and is the one appearing in the representation \rf{V^0} of the value function $V_\infty^0(\cd)$ of Problem {\rm(SLQ)$^0_\infty$}.

\end{proposition}

\begin{proof} From \rf{<P}, we see that
$$P_\infty\deq\lim_{t\to\i}\Si(t)\les P$$
exists and is positive definite. To see that $P_\infty$ satisfies the ARE \rf{ARE3}, we observe that (by \rf{Ric:Sigma})
$$\Si(T+1)-\Si(T) = \int_T^{T+1}\[\cQ(\Si(t))-\cS(\Si(t))^\top\cR(\Si(t))^{-1}\cS(\Si(t))\] dt.$$
Letting $T\to\infty$ yields \rf{ARE3}. Finally, we observe that \rf{ARE3} can be written as
$$ P[A\!+\!B\cK(P)] + [A\!+\!B\cK(P)]^\top\! P + [C\!+\!D\cK(P)]^\top\! P[C\!+\!D\cK(P)]
   + Q + \cK(P)^\top\! R\cK(P) = 0. $$
Since $Q,R>0$, the above implies
$$P[A+B\cK(P)]+[A+B\cK(P)]^\top P+[C+D\cK(P)]^\top P[C+D\cK(P)]<0.$$
Since $P>0$, we conclude from \autoref{equivalence} that $\cK(P)$ is a stabilizer of $[A,C;B,D]$.
\end{proof}

An interesting further issue is how fast $\Si(t)$ converges to the solution $P$ of \rf{ARE3} as $t\to\i$. To address this issue, we need the following lemma.

\begin{lemma}\label{lmm:Pi-stable}
Suppose that the system $[A,C]$ is $L^2$-stable and let $\b$ be the constant in \rf{e-stable}.
Let $f:[0,\i)\times\dbR^{n\times n}\to\dbR^{n\times n}$ be a continuous function satisfying $f(t,0)=0$ and
$$ |f(t,M)-f(t,N)| \les\rho|M-N|(|M|+|N|), \q \forall t\ges0,~\forall M,N\in\dbR^{n\times n}, $$
for some constant $\rho>0$. Then for small initial state $\varPi_0\in\dbR^{n\times n}$, the ODE
\begin{equation*}\left\{\begin{aligned}
\dot\varPi(t) &= \varPi(t) A + A^\top\varPi(t) + C^\top\varPi(t) C + f(t,\varPi(t)), \q t\in[0,\i), \\
    \varPi(0) &= \varPi_0,
\end{aligned}\right.\end{equation*}
has a unique exponentially stable solution with decay rate $\b$, i.e., for some constant $\d>0$,
$$|\varPi(t)|\les\d e^{-\b t},\q\forall t\ges0.$$
\end{lemma}

\begin{proof} Let $\d>0$ be an undetermined constant, and consider the complete metric space (with respect to the uniform metric)
$$\sM=\big\{M(\cd)\in C([0,\i);\dbR^{n\times n});|M(t)|\les\d e^{-\b t},~\forall t\ges0\big\}.$$
Clearly, for each $M(\cd)\in\sM$ and each initial state $\varPi_0\in\dbR^{n\times n}$, the ODE
$$\left\{\2n\ba{ll}
\ds\dot\varPi(t)=\varPi(t)A+A^\top\varPi(t)+C^\top\varPi(t)C+ f(t,M(t)),\q t\in[0,\i),\\
\ds\varPi(0)=\varPi_0\ea\right.$$
has a unique solution $\varPi(\cd)\equiv\sT[M(\cd)]$.
Let $\F(\cd)$ be the solution of \rf{SDE:Phi0}.
By It\^{o}'s rule, we have for $0\les s\les t<\i$,
$$\ba{ll}
\ds d\big[\F(s)^\top\varPi(t-s)\F(s)\big]=-\,\F(s)^\top f(t-s,M(t-s))\F(s) ds\\
\ds\qq\qq\qq\qq\qq\qq+\,\F(s)^\top\big[C^\top\varPi(t-s) + \varPi(t-s)C\big]\F(s)dW(s).\ea$$
Consequently,
$$\ba{ll}
\ds\varPi(t-s)=\dbE\Big\{\[\F(t)\F(s)^{-1}\]^\top \varPi(0)\[\F(t)\F(s)^{-1}\] \\
\ns\ds\qq\qq\qq+\int_s^t \[\F(r)\F(s)^{-1}\]^\top f(t-r,M(t-r)) \[\F(r)\F(s)^{-1}\] dr\Big\}.\ea$$
Taking $s=0$ yields
\bel{Pi-formula}\ba{ll}
\ns\ds\varPi(t)=\dbE\[\F(t)^\top\varPi_0\F(t)\]+\dbE\int_0^t \F(r)^\top f(t-r,M(t-r))\F(r)dr\\
\ns\ds=\dbE\[\F(t)^\top\varPi_0\F(t)\]+\dbE\int_0^t\F(t-s)^\top f(s,M(s))\F(t-s)ds.\ea\ee
It follows from \rf{e-stable} and the assumption on $f$ that
$$\ba{ll}
\ns\ds|\varPi(t)|\les|\varPi_0|\a e^{-\b t}+\int_0^t\rho\a e^{-\b(t-s)}|M(s)|^2ds\\
\ns\ds\qq\q\les|\varPi_0|\a e^{-\b t}+\int_0^t\rho\a \d^2e^{-\b(t+s)}ds\les\a\lt(|\varPi_0|+{\rho\d^2\over\b}\rt)e^{-\b t}.\ea$$
Take $\d\in(0,{\b\over2\a\rho}]$. Then
$$\a\lt({\d\over2\a}+{\rho\d^2\over\b}\rt)\les\d.$$
Thus, for $|\varPi_0|\les{\d\over2\a}$, $\sT$ maps $\sM$ into itself. Further, for any $M(\cd),N(\cd)\in\sM$, we have from \rf{Pi-formula} that
$$\ba{ll}
\ns\ds|\sT[M(\cd)](t)-\sT[N(\cd)](t)|\les\a\int_0^t e^{-\b(t-s)}|f(s,M(s))-f(s,N(s))|ds\\
\ns\ds\qq\qq\qq\qq\qq\q~\les\a\rho\int_0^t e^{-\b(t-s)}|M(s)-N(s)|\(|M(s)|+|N(s)|\) ds \\
\ns\ds\qq\qq\qq\qq\qq\q~\les 2\a\rho\d\int_0^te^{-\b t}|M(s)-N(s)|ds\\
\ns\ds\qq\qq\qq\qq\qq\q~\les \(2\a\rho\d te^{-\b t}\) \sup_{s\ges0}|M(s)-N(s)|.\ea$$
If we take $\d>0$ small enough so that also $2\a\rho\d te^{-\b t}\les1/2$ for all $t\ges0$, then $\sT$ is a contraction mapping on $\sM$. The desired result therefore follows.
\end{proof}

The following results shows the rate of convergence of $\Si(t)$ is exponential.

\begin{theorem}\label{thm:Pi-estable} Let {\rm\ref{TP:H1}--\ref{TP:H2}} hold. There exist positive constants $K,\l>0$ such that
\begin{align}\label{P-Si:stable}
|P-\Si(t)| \les Ke^{-{2\l}t}, \q\forall t\ges0.
\end{align}
\end{theorem}

\begin{proof}
Let $\varPi(t)=P-\Si(t)$. Then
\begin{align*}
\dot\varPi &= \varPi A + A^\top\varPi + C^\top\varPi C - \cS(P)^\top\cR(P)^{-1}\cS(P) + \cS(\Si)^\top\cR(\Si)^{-1}\cS(\Si) \\
&= \varPi[A+B\cK(P)] + [A+B\cK(P)]^\top\varPi + [C+D\cK(P)]^\top\varPi[C+D\cK(P)] \\
&\hp{=\ } -\cS(\varPi)^\top\cK(P) - \cK(P)^\top\cS(\varPi) - \cK(P)^\top D^\top\varPi D\cK(P) \\
&\hp{=\ } +\cK(P)^\top\cS(P) + \cS(\Si)^\top\cR(\Si)^{-1}\cS(\Si)  \\
&= \varPi[A+B\cK(P)] + [A+B\cK(P)]^\top\varPi + [C+D\cK(P)]^\top\varPi[C+D\cK(P)] \\
&\hp{=\ } -[\cS(\varPi) + D^\top\varPi D\cK(P)]^\top\cK(P) + [\cK(P)^\top + \cS(\Si)^\top\cR(\Si)^{-1}]\cS(\Si).
\end{align*}
On the other hand,
\begin{align*}
& [\cK(P)^\top + \cS(\Si)^\top\cR(\Si)^{-1}]\cS(\Si) - [\cS(\varPi) + D^\top\varPi D\cK(P)]^\top\cK(P) \\
&\q= [\cK(P)^\top + \cS(\Si)^\top\cR(\Si)^{-1}]\cS(\Si) + [\cS(\varPi) + D^\top\varPi D\cK(P)]^\top\cR(P)^{-1}[\cS(\varPi)+\cS(\Si)]  \\
&\q= [-\cS(P)^\top + \cS(\Si)^\top\cR(\Si)^{-1}\cR(P) + \cS(\varPi)^\top + \cK(P)^\top D^\top\varPi D]\cR(P)^{-1}\cS(\Si) \\
&\q\hp{=\ }+[\cS(\varPi) + D^\top\varPi D\cK(P)]^\top\cR(P)^{-1}\cS(\varPi)  \\
&\q= \big[\cS(\Si)^\top\cR(\Si)^{-1}D^\top\varPi D + \cK(P)^\top D^\top\varPi D\big] \cR(P)^{-1}\cS(\Si) \\
&\q\hp{=\ }+[\cS(\varPi) + D^\top\varPi D\cK(P)]^\top\cR(P)^{-1}\cS(\varPi)  \\
&\q= \big[-\cK(\Si) + \cK(P)\big]^\top D^\top\varPi D\cR(P)^{-1}\cS(\Si)+[\cS(\varPi) + D^\top\varPi D\cK(P)]^\top\cR(P)^{-1}\cS(\varPi)  \\
&\q= -\big[\cS(\varPi)+D^\top\varPi D\cK(\Si)\big]^\top\cR(P)^{-1}D^\top\varPi D\cR(P)^{-1}\cS(\Si)  \\
&\q\hp{=\ }+[\cS(\varPi) + D^\top\varPi D\cK(P)]^\top\cR(P)^{-1}\cS(\varPi).
\end{align*}
Set $\cA\deq A+B\cK(P)$, $\cC\deq C+D\cK(P)$, and
\begin{equation}\label{def:f(t,Pi)}\begin{aligned}
& f(t,\varPi) \deq [\cS(\tb{\varPi}) + D^\top\tb{\varPi} D\cK(P)]^\top\cR(P)^{-1}\cS(\tb{\varPi}) \\
&\hp{f(t,\varPi)=\ } -[\cS(\tb{\varPi}) + D^\top\tb{\varPi} D\cK(\Si(t))]^\top\cR(P)^{-1}D^\top\tb{\varPi} D\cR(P)^{-1}\cS(\Si(t)).
\end{aligned}\end{equation}
Then we can rewrite the equation for $\varPi(\cd)$ as follows:
$$ \dot\varPi(t) = \varPi(t)\cA + \cA^\top\varPi(t) + \cC^\top\varPi(t)\cC + f(t,\varPi(t)). $$
From \autoref{prop:sol-ARE} we know that the system $[\cA,\cC]$ is $L^2$-stable. Thus, by \autoref{equivalence}, there exist constants $K,\l>0$
such that the solution $\Psi(\cd)$ to
\begin{equation}\label{}\left\{\begin{aligned}
d\Psi(t) &= \cA\Psi(t)dt + \cC\Psi(t)dW(t), \q t\ges0, \\
 \Psi(0) &= I_n
\end{aligned}\right.\end{equation}
satisfies
\begin{align}\label{Psi:stable}
  \dbE|\Psi(t)|^2 \les K e^{-2\l t}, \q\forall t\ges0.
\end{align}
Also, it is easy to see that the function defined by \rf{def:f(t,Pi)} satisfies the properties
in \autoref{lmm:Pi-stable}.
Since $\ds\lim_{t\to\i}\varPi(t)=0$, we conclude from \autoref{lmm:Pi-stable} that \rf{P-Si:stable}
holds for large $t$ and hence all $t\ges0$ (with a possibly different constant $K>0$).
\end{proof}

\section{The Turnpike Property}\label{Sec:TP}

Let $(\bar X_{\scT}(\cd),\bar u_{\scT}(\cd))$ be the optimal pair of Problem (SLQ)$_{\scT}$ for the given initial state $x$ and $(\bar Y_{\scT}(\cd),\bar Z_{\scT}(\cd))$ the adapted solution
to the corresponding adjoint equation in \rf{os:ProbT}.
Let $(x^*,u^*)$ be the unique solution of Problem (O) and $\l^*\in\dbR^n$ the corresponding Lagrange multiplier. Define
\bel{def:hX+hu}
\h X_{\scT}(\cd)=\bar X_{\scT}(\cd)-x^*,\q\h u_{\scT}(\cd) = \bar{u}_{\scT}(\cd)-u^*,\q\h Y_{\scT}(\cd) = \bar{Y}_{\scT}(\cd)-\l^*.\ee
We are now ready to state the main result of this paper, which establishes the exponential turnpike
property of Problem (SLQ)$_{\scT}$ as well as of the adjoint process.

\begin{theorem}\label{thm:e-turnpike-EX}
Let {\rm\ref{TP:H1}--\ref{TP:H2}} hold. Then there exist positive constants $K,\mu>0$, independent of $T$, such that
\bel{hX+hu+hY<Lmu}
\big|\dbE[\h X_{\scT}(t)]\big|+\big|\dbE[\h u_{\scT}(t)]\big|+ \big|\dbE[\h Y_{\scT}(t)]\big|\les K\big[e^{-\mu t}+e^{-\mu(T-t)}\big],\qq\forall t\in[0,T].\ee
\end{theorem}

As an immediate consequence of \autoref{thm:e-turnpike-EX}, we have the following corollary, which shows
that the integral and the mean-square turnpike properties also hold for Problem (SLQ)$_{\scT}$.

\begin{corollary}\label{crllry:turnpike-intgrl}
Let {\rm\ref{TP:H1}--\ref{TP:H2}} hold. Then as $T\to\i$,
$$\ba{ll}
\ns\ds{1\over T}\int_0^T\dbE[\bX_{\scT}(t)]dt\to x^*,\qq{1\over T}\dbE\Big|\int_0^T\(\bX_{\scT}(t)-x^*\)dt\Big|^2\to0, \\
\ns\ds{1\over T}\int_0^T \dbE[\bu_{\scT}(t)]dt \to u^*, \qq\, {1\over T}\dbE\Big|\int_0^T\(\bu_{\scT}(t)-u^*\)dt\Big|^2 \to0.\ea$$
\end{corollary}

In order to prove \autoref{thm:e-turnpike-EX}, let us first make some observations. For convenience, let us rewrite the optimality system \rf{os:ProbT} of Problem (SLQ)$_{\scT}$ and the characterization \rf{L-conditions} of the optimal solution to Problem (O) together here
\bel{5.3}\left\{\2n\ba{ll}
\ds d\bar X_{\scT}(t)=[A\bar X_{\scT}(t)+B\bar u_{\scT}(t)+b]dt+[C\bar X_{\scT}(t)+D\bar u_{\scT}(t)+\si]dW(t),\\
\ns\ds d\bar Y_{\scT}(t)=-\big[A^\top\bar Y_{\scT}(t)+C^\top\bar Z_{\scT}(t)+Q\bar X_{\scT}(t)+q\big]dt+\bar Z_{\scT}(t)dW(t),\\
\ns\ds\bar X_{\scT}(0)=x,\q\bar Y_{\scT}(T)=0,\\
\ns\ds B^\top\bar Y_{\scT}(t)+D^\top\bar Z_{\scT}(t)+R\bar u_{\scT}(t)= 0, \q\ae~t\in[0,T],~\as,\ea\right.\ee
and
\bel{5.4}\left\{\2n\ba{ll}
\ds Qx^*+A^\top\l^*+C^\top P(Cx^*+Du^*+\si)+q=0,\\
\ds Ru^*+B^\top\l^*+D^\top P(Cx^*+Du^*+\si)=0.\ea\right.\ee
Noting that $(x^*,u^*)\in\sV$, we have
$$Ax^*+Bu^*+b=0.$$
Also, we denote
$$\si^*=Cx^*+Du^*+\si.$$
Then, a direction calculation yields the following:
\bel{hX}\left\{\2n\ba{ll}
\ds d\h X_{\scT}(t)=\big[A\h X_{\scT}(t)+B\h u_{\scT}(t)\big]dt+ \big[C\h X_{\scT}(t)+D\h u_{\scT}(t)+\si^*]dW(t),\\
\ns\ds d\h Y_{\scT}(t)=-\big[A^\top\h Y_{\scT}(t)+C^\top\bar Z(t) +Q\h X_{\scT}(t)-C^\top P(t)\si^*\big]dt+\bar Z(t)dW,\\
\ns\ds\h X_{\scT}(0) = x-x^*,\q\h Y_{\scT}(T)=-\l^*,\\
\ns\ds B^\top\h Y_{\scT}(t)+D^\top\bar Z(t)+R\h u_{\scT}(t)-D^\top P(t)\si^*=0,\q\ae~t\in[0,T],~\as\ea\right.\ee
Comparing the above with \rf{SLQ:opti-sys}--\rf{SLQ:sta-cdtn} in \autoref{lmm:opti-sys}(ii), we see that $(\h X_{\scT}(\cd),\h u_{\scT}(\cd))$ is an optimal pair of the stochastic LQ problem
with state equation
$$\left\{\2n\ba{ll}
\ds dX(t)=\big[AX(t)+Bu(t)\big]dt+\big[CX(t)+Du(t)+\si^*\big]dW(t),\\
\ds X(0)=x-x^*,\ea\right.$$
and cost functional
$$\ba{ll}
\ds J(x;u)=\dbE\Big\{-2\lan\l^*,X(T)\ran+\int_0^T\[\lan QX(t),X(t)\ran +\lan Ru(t),u(t)\ran\\
\ns\ds\qq\qq\qq\qq\qq\qq\qq\qq-2\lan C^\top P(t)\si^*,X(t)\ran-2\lan D^\top P(t)\si^*,u(t)\ran\] dt\Big\}.\ea$$
Now applying \autoref{lmm:opti-sys}(iv), we obtain the following result immediately.

\begin{proposition} Let {\rm\ref{TP:H1}--\ref{TP:H2}} hold. Let $P_{\scT}(\cd)$ be the solution to \rf{Ric:PT} and
\bel{def:ThT}\Th_{\scT}(t)\deq\cK(P_{\scT}(t))=
-\cR(P_{\scT}(t))^{-1}\cS(P_{\scT}(t)),\ee
and let $\f_{\scT}(\cd)$ be the solution to the ODE
\bel{BODE:f}\left\{\2n\ba{ll}
\ds\dot\f_{\scT}(t)+[A+B\Th_{\scT}(t)]^\top\f_{\scT}(t)+ [C+D\Th_{\scT}(t)]^\top(P_{\scT}(t)-P)\si^*,\\
\ns\ds\f_{\scT}(T)=-\l^*.\ea\right.\ee
Then the process $\h u_{\scT}(\cd)$ defined in \rf{def:hX+hu} is given by
\bel{rep:hu}\h u_{\scT}(t)=\Th_{\scT}(t)\h X_{\scT}(t)- \cR(P_{\scT}(t))^{-1}[B^\top\f_{\scT}(t)+D^\top(P_{\scT}(t)-P)\si^*].\ee
\end{proposition}

To prove \autoref{thm:e-turnpike-EX}, we also need the following lemma.

\begin{lemma} Let {\rm\ref{TP:H1}--\ref{TP:H2}} hold. The solution $\f_{\scT}(\cd)$ to the ODE \rf{BODE:f} satisfies
\bel{phi:guji}|\f_{\scT}(t)|\les Ke^{-\l(T-t)},\q\forall\,0\les t\les T,\ee
for some constants $K,\l>0$ independent of $T$.
\end{lemma}

\begin{proof} For notational simplicity, we let
\bel{notation:ThcAcC}\Th\deq\cK(P)=-(R+D^\top\1n PD)^{-1}(B^\top\1n P+D^\top\1n PC),\q\cA\deq A+B\Th,\q\cC \deq C+D\Th\ee
and write \rf{BODE:f} as
\bel{BODE:f*}\left\{\2n\ba{ll}
\ds\dot\f_{\scT}(t)=-\cA^\top\f_{\scT}(t)-[B(\Th_{\scT}(t)
-\Th)]^\top\f_{\scT}(t)-[C+D\Th_{\scT}(t)]^\top(P_{\scT}(t)-P)\si^*, \\
\ds\f_{\scT}(T)=-\l^*.\ea\right.\ee
By the variation of constants formula,
\bel{fT=}\f_{\scT}(t)=e^{\cA^\top(T-t)}\[-\l^*+\int_t^T e^{\cA^\top(s-T)}\rho(s)ds\],\ee
where
$$ \rho(s) = [B(\Th_{\scT}(s)-\Th)]^\top\f_{\scT}(s) + [C+D\Th_{\scT}(s)]^\top(P_{\scT}(s)-P)\si^*. $$
Recall from the proof of \autoref{thm:Pi-estable} that there exist constants $K,\l>0$,
independent of $T$, such that \rf{P-Si:stable} and \rf{Psi:stable} hold,
and note that $\Si(T-t)=P_{\scT}(t)$ (\autoref{prop:Si-increase}). We have
\begin{align}\label{PT-P:guji}
 \big|e^{\cA^\top t}\big| \les Ke^{-\l t}, \q |P_{\scT}(t)-P| \les Ke^{-{2\l}(T-t)}, \q\forall\, 0\les t\les T<\i,
\end{align}
where and hereafter, $K$ represents a generic constant (independent of $T$)
which can be different from line to line,
but $\l$ is the fixed constant in \rf{P-Si:stable} and \rf{Psi:stable}. Observe that
$$\ba{ll}
\ds\Th_{\scT}(s)-\Th=\cR(P)^{-1}\cS(P)-\cR(P_{\scT}(s))^{-1}
\cS(P_{\scT}(s))\\
\ns\ds\qq\qq\q=\cR(P)^{-1}\cS(P-P_{\scT}(s))+\[\cR(P)^{-1}
-\cR(P_{\scT}(s))^{-1}\]\cS(P_{\scT}(s))\\
\ns\ds\qq\qq\q=\cR(P)^{-1}\cS(P-P_{\scT}(s))+\cR(P)^{-1}D^\top [P_{\scT}(s)-P]D\cR(P_{\scT}(s))^{-1}\cS(P_{\scT}(s)).\ea$$
Since $P_{\scT}(\cd)$, and hence $\Th_{\scT}(\cd)$, is bounded uniformly in $T$, we have
\bel{ThT-Th}|\Th_{\scT}(s)-\Th|\les Ke^{-2\l(T-s)},\q\forall\, 0\les s\les T<\i.\ee
It follows that
$$|\rho(s)|\les Ke^{-2\l(T-s)}\big[|\f_{\scT}(s)|+1\big].$$
If we let $h(t)=e^{\l(T-t)}|\f_{\scT}(t)|$, then by \rf{fT=},
$$\ba{ll}
\ds h(t)\les K\Big[|\l^*|+\int_t^TKe^{-\l(s-T)}|\rho(s)|ds\Big]\\
\ns\ds\les K+K\int_t^T\big[e^{-\l(T-s)}|\f_{\scT}(s)|+e^{-\l(T-s)}\big]ds\\
\ds=K+K\int_t^T\big[e^{-2\l(T-s)}h(s)+e^{-\l(T-s)}\big]ds.\ea$$
Applying Gronwall's inequality we obtain that for some constant $K>0$, independent of $T>0$,
$$h(t)\les K,\q\forall t\in[0,T].$$
The desired result then follows.
\end{proof}

\begin{proof}[Proof of \autoref{thm:e-turnpike-EX}]
For notational simplicity, we let
$$\h v_{\scT}(t)\deq -\cR(P_{\scT}(t))^{-1}[B^\top\f_{\scT}(t)+D^\top(P_{\scT}(t)-P)\si^*].$$
Substituting \rf{rep:hu} into the state equation for $\h X_{\scT}(\cd)$ yields
$$\left\{\2n\ba{ll}
\ds d\h X_{\scT}(t)=\big\{[A+B\Th_{\scT}(t)]\h X_{\scT}(t)+B\h v_{\scT}(t)\big\}dt \\
\ns\ds\qq\qq\qq+\big\{[C+D\Th_{\scT}(t)]\h X_{\scT}(t)+D\h v_{\scT}(t)+\si^*\big\}dW(t),\q t\in[0,T],\\
\ns\ds\h X_{\scT}(0)=x-x^*.\ea\right.$$
Using the notation \rf{notation:ThcAcC}, we can rewrite the above as
\bel{hX:cloop}\left\{\2n\ba{ll}
\ds d\h X_{\scT}(t)=[\cA\h X_{\scT}(t)+\xi(t)]dt+[\cC\h X_{\scT}(t) +\eta(t)]dW(t),\q t\in[0,T],\\
\ns\ds\h X_{\scT}(0)=x-x^*\equiv\h x,\ea\right.\ee
where
$$\xi(t)=B[\Th_{\scT}(t)-\Th]\h X_{\scT}(t)+B\h v_{\scT}(t),\q \eta(t)=D[\Th_{\scT}(t)-\Th]\h X_{\scT}(t)+D\h v_{\scT}(t)+\si^*.$$
Taking expectations in \rf{hX:cloop}, we get
\bel{ODE:EhX}\left\{\2n\ba{ll}
\ns\ds d\dbE[\h X_{\scT}(t)]=\big\{\cA\dbE[\h X_{\scT}(t)]+B[\Th_{\scT}(t)-\Th]\dbE[\h X_{\scT}(t)] +B\h v_{\scT}(t)\big\}dt,\q t\in[0,T],\\
\ns\ds\dbE[\h X_{\scT}(0)]=\h x.\ea\right.\ee
Noting
$$P\cA+\cA^\top P+\cC^\top P\cC+\Th^\top R\Th+Q=0,$$
we obtain
$$\ba{ll}
\ds\lan P\dbE[\h X_{\scT}(t)],\dbE[\h X_{\scT}(t)]\ran-\lan P\h x,\h x\ran\1n=\2n\int_0^t\2n\big\{\1n-\lan(\cC^\top P\cC+\Th^\top R\Th+Q)\dbE[\h X_{\scT}(s)],\dbE[\h X_{\scT}(s)\ran\\
\ns\ds\qq\qq\qq\qq\qq\qq\qq+2\lan P\dbE[\h X_{\scT}(s)],B[\Th_{\scT}(s)-\Th]\dbE[\h X_{\scT}(s)] +B\h v_{\scT}(s)\ran\big\}ds.\ea$$
Further, since $P>0$ and $\cC^\top P\cC+\Th^\top R\Th+Q>0$, we have from the above that
$$\ba{ll}
\ds|\dbE[\h X_{\scT}(t)]|^2\les\a_1+\dbE\int_0^t\[-\a_2|\dbE[\h X_{\scT}(s)]|^2+\a_1|\Th_{\scT}(s)-\Th|\,|\dbE[\h X_{\scT}(s)]|^2\\
\ns\ds\qq\qq\qq+2\a_1|\dbE[\h X_{\scT}(s)]|\cd|\h v_{\scT}(s)|\]ds,\ea$$
for some constants $\a_1,\a_2>0$ independent of $T$. Using the Cauchy-Schwarz inequality we can obtain that with two possibly different constants $\a_1,\a_2>0$,
$$|\dbE[\h X_{\scT}(t)]|^2
\1n\les\1n\a_1\1n+\1n\dbE\2n\int_0^t\2n\big[-\a_2|\dbE[\h X_{\scT}(s)]|^2+\a_1|\Th_{\scT}(s)-\Th|\,|\dbE[\h X_{\scT}(s)]|^2+\a_1|\hv_{\scT}(s)|^2\]ds.$$
Recalling \rf{phi:guji} and \rf{ThT-Th}, we see that with another two possibly different constants
$\a_1,\a_2>0$ independent of $T$,
$$|\dbE[\h X_{\scT}(t)]|^2\les\a_1+\dbE\int_0^t\[\big(\a_1e^{-2\l(T-s)}-\a_2\big)|
\dbE[\hX_{\scT}(s)]|^2 + \a_1e^{-2\l(T-s)}\]ds.$$
For any $0\les s\les t\les T$, we have
$$\ba{ll}
\ds\phi(s,t)\deq \exp\Big\{\int_s^t\big(\a_1e^{-2\l(T-r)}-\a_2\big)dr\Big\}\\
\ns\ds\qq\q=\exp\lt\{{\a_1\over2\l}\big[e^{-2\l(T-t)}-e^{-2\l(T-s)}\big]
-\a_2(t-s)\rt\}\les e^{\a_1\over2\l}e^{-\a_2(t-s)},\ea$$
and thus
$$\int_0^t\phi(s,t)e^{-2\l(T-s)}ds\les e^{\a_1\over2\l}\int_0^te^{-2\l(T-s)}ds
\les {1\over2\l}e^{\a_1\over2\l} e^{-2\l(T-t)}.$$
Hence, by Gronwall's inequality,
\bel{hX<Lmu}\ba{ll}
\ns\ds|\dbE[\h X_{\scT}(t)]|^2\les\a_1\phi(0,t)+ \a_1\int_0^t\phi(s,t)e^{-2\l(T-s)}ds\\
\ns\ds\qq\qq\q\les\a_1e^{\a_1\over2\l}e^{-\a_2t}\1n+\1n{\a_1\over2\l}
e^{\a_1\over2\l} e^{-2\l(T-t)}\les L\big(e^{-\mu t}+ e^{-\mu(T-t)}\big),\q\forall t\in[0,T],\ea\ee
where $L=\lt(\a_1+{\a_1\over2\l}\rt)e^{\a_1\over2\l}$ and $\mu=\a_2\wedge2\l$.
Now using the relation \rf{rep:hu} and noting that
$$|\hv_{\scT}(t)| \les Ke^{-\l(T-t)}, $$
we can show that
\bel{hu<Lmu}
|\dbE[\h u_{\scT}(t)]|^2
\les 2|\Th_{\scT}(t)|^2\,|\dbE[\h X_{\scT}(t)]|^2+ 2|\h v_{\scT}(t)|^2\les L \[e^{-\mu t} + e^{-\mu(T-t)}\], \q\forall t\in[0,T],\ee
for a possibly different constant $L$ that is independent of $T$.
Finally, we can verify that the following relation holds:
$$\h Y_{\scT}(t)=P_{\scT}(t)\h X_{\scT}(t)+\f_{\scT}(t),
\q \bar Z_{\scT}(t) = P_{\scT}(t)[C\h X_{\scT}(t) + D\h u_{\scT}(t)+\si^*],$$
from which it follows that (with a possibly different constant $L>0$)
\bel{hY<Lmu}
|\dbE[\h Y_{\scT}(t)]|^2\les L\[e^{-\mu t}+e^{-\mu(T-t)}\],\q\forall t\in[0,T].\ee
Combining \rf{hX<Lmu}--\rf{hY<Lmu}, we get the desired \rf{hX+hu+hY<Lmu}.
\end{proof}

We conclude this section by showing that the value function $V_{\scT}(x)$
of Problem (SLQ)$_{\scT}$ converges in the sense of time-average to the
optimal value $V$ of the associated static optimization problem.

\begin{theorem}
Let {\rm\ref{TP:H1}--\ref{TP:H2}} hold. Then
\begin{align}\label{eqn:VTgotoV}
{1\over T}V_{\scT}(x)\to V  \q\text{as } T\to\i.
\end{align}
\end{theorem}

\begin{proof}
We observe first that
$$ V_{\scT}(x) = J_{\scT}(x;\bar u_{\scT}(\cd)) $$
can be written as follows:
\bel{cost:fenjie}\ba{ll}
\ds J_{\scT}(x;\bar u_{\scT}(\cd))={1\over2}\int_0^T\[f(\dbE\bar X_{\scT}(t))+g(\dbE\bar u_{\scT}(t))\]dt\\
\ns\ds\qq\qq\qq\qq+{1\over2}\dbE\int_0^T \[\lan Q\check{X}_{\scT}(t),\check{X}_{\scT}(t)\ran + \lan R\check{u}_{\scT}(t),\check{u}_{\scT}(t)\ran \] dt,\ea\ee
where $f(x)\deq\lan Qx,x\ran+2\lan q,x\ran$, $g(u)\deq\lan Ru,u\ran$, and
$$\check{X}_{\scT}(\cd)\deq \bar X_{\scT}(\cd)-\dbE\bar X_{\scT}(\cd),
\q\check{u}_{\scT}(\cd)\deq \bar u_{\scT}(\cd)-\dbE\bar u_{\scT}(\cd). $$
By \autoref{crllry:turnpike-intgrl},
\begin{align}\label{VT:I}
{1\over T}\int_0^T \[f(\dbE\bar X_{\scT}(t)) + g(\dbE\bar u_{\scT}(t))\] dt \to f(x^*)+g(u^*).
\end{align}
On the other hand, $\check{X}_{\scT}(\cd)$ evolves according to the following SDE:
$$\left\{\ba{ll}
\ds d\check X_{\scT}(t)=[A\check X_{\scT}(t)+B\check u_{\scT}(t)]dt+[C\check X_{\scT}(t)+D\check u_{\scT}(t)+\bar\si(t)]dW,\\
\ns\ds\check X_{\scT}(0)=0,\ea\right.$$
where
$$\bar\si(t)\deq C\dbE\bar X_{\scT}(t)+D\dbE\bar u_{\scT}(t)+\si.$$
Let $P_{\scT}(\cd)$ be the solution to the Riccati equation \rf{Ric:PT}. Then
\bel{PcheckX}\ba{ll}
\ns\ds0=\dbE\lan P_{\scT}(T)\check{X}_{\scT}(T),\check{X}_{\scT}(T)\ran-\dbE\lan P_{\scT}(0)\check{X}_{\scT}(0),\check{X}_{\scT}(0)\ran\\
\ns\ds\q= \dbE\int_0^T \Big\{\lan\dot P_{\scT}(t)\check{X}_{\scT}(t),\check{X}_{\scT}(t)\ran
   +2\lan P_{\scT}(t)\check{X}_{\scT}(t),A\check{X}_{\scT}(t) + B\check{u}_{\scT}(t)\ran\\
\ns\ds\qq\qq+\lan P_{\scT}(t)[C\check{X}_{\scT}(t) + D\check{u}_{\scT}(t) + \bar\si(t)],                C\check{X}_{\scT}(t) + D\check{u}_{\scT}(t) + \bar\si(t)\ran\Big\} dt.\ea\ee
Noting the fact
$$ \dbE\check{X}_{\scT}(t) =0, \q \dbE\check{u}_{\scT}(t)=0, \q\forall t\in[0,T],$$
and using \rf{Ric:PT}, we derive from \rf{PcheckX} that
\begin{align}\label{PcheckX*}
0 &= \dbE\int_0^T \Big\{\blan[\cS(P_{\scT}(t))^\top \cR(P_{\scT}(t))^{-1}\cS(P_{\scT}(t))-Q]
     \check{X}_{\scT}(t),\check{X}_{\scT}(t)\bran \nn\\
&\hp{=\ } +2\lan\check{u}_{\scT}(t),\cS(P_{\scT}(t))\check{X}_{\scT}(t)\ran
          +\lan D^\top P_{\scT}(t)D\check{u}_{\scT}(t),\check{u}_{\scT}(t)\ran
          +\lan P_{\scT}(t)\bar\si(t),\bar\si(t)\ran\Big\} dt.
\end{align}
%
It follows that
$$\ba{ll}
\ns\ds\dbE\int_0^T \[\lan Q\check{X}_{\scT}(t),\check{X}_{\scT}(t)\ran + \lan R\check{u}_{\scT}(t),\check{u}_{\scT}(t)\ran \] dt\\
\ns\ds\q= \dbE\int_0^T \Big\{\blan[\cS(P_{\scT})^\top \cR(P_{\scT})^{-1}\cS(P_{\scT})]\check{X}_{\scT},\check{X}_{\scT}\bran
     +2\lan\check{u}_{\scT},\cS(P_{\scT})\check{X}_{\scT}\ran\\
\ns\ds\qq\qq+\lan\cR(P_{\scT})\check{u}_{\scT},\check{u}_{\scT}\ran +\lan P_{\scT}\bar\si,\bar\si\ran\Big\} dt\\
\ns\ds\q= \dbE\int_0^T \Big\{\lan\cR(P_{\scT})[\check{u}_{\scT}-\Th_{\scT}\check{X}_{\scT}],\check{u}_{\scT}-\Th_{\scT}\check{X}_{\scT}\ran
     +\lan P_{\scT}\bar\si,\bar\si\ran\Big\} dt.\ea$$
By \rf{rep:hu}, we have
$$\ba{ll}
\ds\check{u}_{\scT}(t) - \Th_{\scT}(t)\check{X}_{\scT}(t)
= \(\hu_{\scT}(t)-\dbE[\hu_{\scT}(t)]\) - \Th_{\scT}(t)\(\hX_{\scT}(t)-\dbE[\hX_{\scT}(t)]\) \\
\ns\ds\qq\qq\qq\qq\q= \(\hu_{\scT}(t)-\Th_{\scT}(t)\hX_{\scT}(t)\) - \dbE\(\hu_{\scT}(t)-\Th_{\scT}(t)\hX_{\scT}(t)\)= 0,\ea$$
and by \rf{P-Si:stable} and \rf{hX+hu+hY<Lmu}, we have
$${1\over T}\int_0^T \lan P_{\scT}(t)\bar\si(t),\bar\si(t)\ran dt \to \lan P\si^*,\si^*\ran, \q\text{as } T\to\i.$$
Therefore, as $T\to\i$
\bel{VT:II-end}{1\over T}\,\dbE\int_0^T \[\lan Q\check{X}_{\scT}(t),\check{X}_{\scT}(t)\ran + \lan R\check{u}_{\scT}(t),\check{u}_{\scT}(t)\ran \] dt
\to \lan P\si^*,\si^*\ran.\ee
Combining \rf{VT:I} and \rf{VT:II-end} yields  \rf{eqn:VTgotoV}.
\end{proof}

\ms

We point out that for general situation, namely, the state equation \rf{state} and the cost functional \rf{TP:cost}, we may carry out the procedure (with more complicated notations) to get the same results, or transform back from the results for reduced problem. The general condition ensuring the results are the stabilizability of the sysmte $[A,C;B,D]$ and the stronger standard condition \rf{R>0} for the weighting functions of the cost functional.

\section{Concluding Remarks}\label{Sec:Concludng}

For linear quadratic stochastic optimal control problems in finite horizons, we have established the turnpike property under natural conditions of stabiizability of the controlled linear SDE, and the strong standard condition of the quadratic cost functional. The crucial contribution of the current paper is to find the correct form of the corresponding static optimization problem in which the diffusion part of the state equation should be getting into the cost functional, rather than taking it to be as an additional equality constraint. Such an idea should have big impact on the study of turnpike type problems for general stochastic optimal control problems. We will report some further results along this line in our future publications.

\end{document}